\documentclass{amsart}
\usepackage{amsmath}
\usepackage{amsfonts}
\usepackage{amssymb}
\usepackage[utf8]{inputenc}
\usepackage{amssymb}
\usepackage{mathrsfs}
\usepackage{amsthm}

\usepackage{enumerate}

\usepackage{graphicx}

\newtheorem{thrm}{Theorem}[section]
\newtheorem{cor}[thrm]{Corollary}
\newtheorem{lemma}[thrm]{Lemma}
\newtheorem{prop}[thrm]{Proposition}

\newtheorem{obs}[thrm]{Observation}

\theoremstyle{definition}
\newtheorem{defin}[thrm]{Definition}
\newtheorem{rem}[thrm]{Remark}

\newtheorem*{xrem}{Remark}

\DeclareMathOperator{\supp}{supp}

\newcommand{\NN}{\mathbb{N}}
\newcommand{\RR}{\mathbb{R}}

\newcommand{\UU}{\mathcal{U}}

\newcommand{\la}{\langle}
\newcommand{\ra}{\rangle}

\begin{document}

\title[Linear homeomorphisms of function spaces]{Linear homeomorphisms of function spaces and the position of a space in its compactification}

\author[M. Krupski]{Miko\l aj Krupski}
\thanks{The author was partially supported by the NCN (National Science Centre, Poland) research Grant no. 2020/37/B/ST1/02613}
\date{}
\address{
Institute of Mathematics\\ University of Warsaw\\ ul. Banacha 2\\
02--097 Warszawa, Poland }
\email{mkrupski@mimuw.edu.pl}

\begin{abstract}
An old question of Arhangel'skii asks if the Menger property of a Tychonoff space $X$ is preserved by homeomorphisms of the space $C_p(X)$ of continuous real-valued functions on $X$ endowed with the pointwise topology.
We provide affirmative answer in the case of linear homeomorphisms. To this end,
we develop a method of studying invariants of linear homeomorphisms of function spaces $C_p(X)$ by looking at the way $X$ is positioned in its
(\v{C}ech-Stone) compactification.  
\end{abstract}

\subjclass[2020]{Primary: 46E10, 54C35, 54D20, 54D40, Secondary:54C60, 54B20, 91A44}

\keywords{Function space, pointwise convergence topology, $C_p(X)$ space, Menger space, Hurewicz space, projective Menger property,
projective Hurewicz property, lower semi-continuous map, $l$-equivalence, linear
homeomorphism, Porada game}

\maketitle

\section{Introduction}
The present paper is concerned mainly with two classical covering-type properties of a topological space $X$, the Menger property and the Hurewicz property,
and their connections with the linear-topological structure of the space $C_p(X)$ of continuous real-valued functions on $X$ equipped with the pointwise topology.
All spaces under consideration are assumed to be Tychonoff.

An old question of Arhangel'skii (cf. \cite[Problem II.2.8]{Arh} or \cite[Problem 4.2.12]{Tkachuk})
asks if the Menger property\footnote{There is inconsistency in the terminology that is used in the literature.
What we (and most of the modern authors) call the Menger property by some authors is called the Hurewicz property.
In this paper by the Hurewicz property we mean something else.} of a space $X$ is preserved by homeomorphisms of its function space $C_p(X)$.
One of the main results of this paper is the following theorem which settles this question in the case of linear homeomorphisms.

\begin{thrm}\label{theorem Menger}
Suppose that $C_p(X)$ and $C_p(Y)$ are linearly homeomorphic. Then $X$ is a Menger space if and only if $Y$ is a Menger space.
\end{thrm}

Let us recall that a topological space $X$ is \textit{Menger} (resp., \textit{Hurewicz})
if for every sequence $(\mathscr{U}_n)_{n\in \NN}$ of open covers of $X$, there is a sequence $(\mathscr{V}_n)_{n\in\NN}$
such that for every $n$, $\mathscr{V}_n$ is a finite subfamily of $\mathscr{U}_n$ and the family $\bigcup_{n\in\NN} \mathscr{V}_n$ covers $X$
(resp., every point of $X$ is contained in $\bigcup\mathscr{V}_n$ for all but finitely many $n$'s).
This classical notions go back to early works of Witold Hurewicz and Karl Menger. Since then they were studied by many authors and found numerous
applications (see \cite{SS} and references
therein).
Clearly, $$\sigma\text{-compact}\Rightarrow\text{Hurewicz}\Rightarrow \text{Menger}\Rightarrow \text{Lindel\"of}.$$
It is known that no implication above is reversible.
There 
has been a lot of work done on
the interplay between the linear topological structure of a function space $C_p(X)$ and topological properties of underlying space $X$;
we refer the interested reader to the monograph \cite{Tkachuk}.
One of the major results in this area of research is the following deep theorem of Velichko \cite{V} (the theorem below was further generalized by Bouziad \cite{B1} to arbitrary Lindel\"of numbers):

\begin{thrm}(Velichko)\label{theorem Velichko}
Suppose that $C_p(X)$ and $C_p(Y)$ are linearly homeomorphic. Then $X$ is Lindel\"of if and only if $Y$ is Lindel\"of.
\end{thrm}

At the other extreme,
it is relatively easy to show that $\sigma$-compactness of $X$ is determined by the linear-topological structure of the function space $C_p(X)$; see, e.g. \cite[6.9.1]{vM}
(actually more is true: $\sigma$-compactness of $X$ can be characterized by a certain topological property of $C_p(X)$; see \cite{Okunev} or
\cite[\S\; III.2]{Arh}). For the Hurewicz property the following theorem was proved by Zdomskyy \cite[Corollary 7]{Zd}:
\begin{thrm}(Zdomskyy)\label{theorem Hurewicz}
Suppose that $C_p(X)$ and $C_p(Y)$ are linearly homeomorphic. Then $X$ is a Hurewicz space if and only if $Y$ is a Hurewicz space.
\end{thrm}
Regarding the Menger property, analogous assertion is provided by our Theorem \ref{theorem Menger}. Though, some partial results were known before.
In \cite{Zd}, Zdomskyy showed that in the linear case, the answer to Arhangel'skii's question mentioned above,
is affirmative under an additional set-theoretic assumption $\mathfrak{u}<\mathfrak{g}$ (see \cite[Corollary 7]{Zd}).
More recently, Sakai \cite{Sakai} gave a partial solution in \textsf{ZFC}\footnote{The abbreviation \textsf{ZFC} stands for "Zermelo–Fraenkel set theory with the axiom of choice"} (see \cite[Theorem 2.5]{Sakai}).

In the proof of Theorem \ref{theorem Menger}, Velichko's Theorem \ref{theorem Velichko} plays an important role.
This is because of the following observation essentially due to Telg\'{a}rsky (see \cite[Proposition 2]{T}, cf. \cite[Proposition 8]{BCM}):

\begin{prop}\label{proposition proj_Menger}
A space $X$ is Menger if and only if $X$ is Lindel\"of and every separable metrizable continuous image of $X$ is Menger.
\end{prop}

Analogous fact is also true for the Hurewicz property (see \cite[Theorem 3.2]{Ko} or \cite[Proposition 31]{BCM})
A space $X$ is called \textit{projectively Menger} (resp., \textit{projectively Hurewicz})
provided every separable metrizable continuous image of $X$ is Menger (resp., Hurewicz).
According to Theorem \ref{theorem Velichko} and Proposition \ref{proposition proj_Menger}, Theorem \ref{theorem Menger}
reduces to the following result which we prove in this paper:
\begin{thrm}\label{thrm proj Menger}
Suppose that $C_p(X)$ and $C_p(Y)$ are linearly homeomorphic. Then $X$ is projectively Menger if and only if $Y$ is projectively Menger.
\end{thrm}
We also prove a similar theorem for the projective Hurewicz property:

\begin{thrm}\label{thrm proj Hurewicz}
Suppose that $C_p(X)$ and $C_p(Y)$ are linearly homeomorphic. Then $X$ is projectively Hurewicz if and only if $Y$ is projectively Hurewicz.
\end{thrm}
From Theorem \ref{thrm proj Hurewicz} we immediately get Zdomskyy's Theorem \ref{theorem Hurewicz} as a corollary.
The above results answer questions asked in \cite{Sakai} by Sakai.

Our approach relies on the fact that that
the (projective) properties of Menger and Hurewicz of a space $X$ can be conveniently expressed in terms of the \v{C}ech-Stone compactification $\beta X$ of $X$. We develop a method of studying invariants of linear homeomorphisms of function spaces $C_p(X)$ by looking at the way $X$ is positioned in its
\v{C}ech-Stone compactification.

\section{Notation and auxiliary results}
In this section we collect some notation and auxiliary results that we shall use throughout the paper.

\subsection{Hyperspaces and set-valued maps}

For a topological space $X$, by $\mathcal{K}(X)$ we denote the set of all nonempty compact subsets of $X$. We endow $\mathcal{K}(X)$ with the \textit{Vietoris topology},
i.e., the topology generated
by basic open sets of the form:
$$\la \UU\ra=\{K\in \mathcal{K}(X):\forall U\in \UU\quad K\cap U\neq \emptyset\; \text{and}\;K\subseteq \bigcup\UU\},$$
where $\UU=\{U_1,\ldots U_n\}$ is a finite collection of open subsets of $X$.

For an integer $n\geq 1$ we put $[X]^{\leq n}=\{K\in \mathcal{K}(X): |K|\leq n\}$ and
$[X]^n=[X]^{\leq n}\setminus [X]^{\leq n-1}$, i.e., $[X]^{\leq n}$ ($[X]^n$) is the subspace of $\mathcal{K}(X)$ consisting of
all at most (precisely) $n$-element subsets of $X$.

A set-valued map $\phi:X\to \mathcal{K}(Y)$ is \textit{lower semi-continuous} if the set $$\phi^{-1}(U)=\{x\in X:\phi(x)\cap U\neq \emptyset\}$$
is open, for every open $U\subseteq Y$.

Let us note the following simple fact.

\begin{lemma}\label{G_delta via lsc map}
For a space $X$ and a compact space $Z$, let
$\phi:X\to \mathcal{K}(Z)$ be a lower semi-continuous map.
If $K$ is a compact $G_\delta$-subset of $Z$, then the set $\phi^{-1}(K)=\{x\in X:\phi(x)\cap K\neq\emptyset\}$
is $G_\delta$ in $X$.
\end{lemma}
\begin{proof}
The set $K$ is compact $G_\delta$, so there is a sequence $U_1, U_2, \ldots$ of open subsets of $Z$ such that
$U_{n+1}\subseteq\overline{U_{n+1}}\subseteq U_n$, for every $n\geq 1$, and
$K=\bigcap_{n=1}^\infty U_n=\bigcap_{n=1}^\infty \overline{U_n}$.
Since the map $\phi$ is lower semi-continuous, it suffices to
check that
$$\phi^{-1}(K)=\bigcap_{n=1}^\infty \phi^{-1}(U_n).$$
Pick $x\in \bigcap_{n=1}^\infty \phi^{-1}(U_n)$, i.e., for every $n\geq 1$, we have $\phi(x)\cap U_n\neq \emptyset$.
The map $\phi$ has compact values and $Z$ is compact. Hence, the intersection of the (decreasing) family $\{\phi(x)\cap \overline{U_n}:n=1,2,\ldots\}$
of nonempty closed subsets of $Z$
must be nonempty, i.e., we have, $\phi(x)\cap\bigcap_{n=1}^\infty \overline{U_n}=\phi(x)\cap K\neq \emptyset$.
This gives $x\in \phi^{-1}(K)$. The converse inclusion is obvious.
\end{proof}

If $\phi:X\to \mathcal{K}(Y)$ is a set-valued map and $A\subseteq X$, then we define the image $\phi(A)$ of $A$ under $\phi$ as $$\phi(A)=\bigcup\{\phi(x):x\in A\}.$$ 

A subset $S$ of $Y$ that meets all values of $\phi:X\to \mathcal{K}(Y)$ is called a \textit{section} of $\phi$.
The following theorem can be attributed to Bouziad (cf. \cite[Theorem 6]{B}):
\begin{thrm}\label{theorem_Bouziad}
Suppose that $X$ is a $G_\delta$ subspace of a compact space. If $C$ is compact, then every lower semi-continuous function $\phi:C\to \mathcal{K}(X)$ admits a compact section. 
\end{thrm}
\begin{proof}
  This is a direct consequence of \cite[Theorem 2]{B} and \cite[Theorem 4.1]{DGL}.
\end{proof}

\subsection{The $k$-Porada game}
Let us recall the description of a certain topological game that will be of great importance in the proof of Theorem \ref{theorem Menger}.
The game defined below is so-called $k$-modification (instead of points one considers compact sets) of a game introduced in \cite{Po}.
It was studied in \cite{T} and, more recently, in \cite{K1}. Our terminology follows \cite{T}.
Let $Z$ be a space and let $X\subseteq Z$ be a subspace of $Z$.

The \textit{$k$-Porada game} on $Z$ with values in $X$ is a game with $\omega$-many innings, played alternately
by two players: I and II.
Player I begins the game and makes the first move by choosing
a pair $(K_0,U_0)$, where $K_0\subseteq X$ is nonempty compact and $U_0$ is an open set in $Z$ that contains $K_0$. Player II responds by choosing an open
(in $Z$) set $V_0$ such that
$K_0\subseteq V_0\subseteq U_0$. In the second round of the game, player I picks a pair $(K_1,U_1)$, where
$K_1$ is a nonempty compact subset of $V_0$ and $U_1$ is an open subset of $Z$ with
$K_1\subseteq U_1\subseteq V_0$. Player II responds by picking an open (in $Z$) set $V_1$ such that $K_1\subseteq V_1\subseteq U_1$. The game continues in this way and stops
after $\omega$ many rounds. Player II wins the game if $\emptyset\neq \bigcap_{n\in\NN} U_n (= \bigcap_{n\in \NN} V_n) \subseteq X$.
Otherwise player I wins.

The game described above is denoted by $kP(Z,X)$.

\subsection{Strategies}

Denote by
$\mathcal{T}_Y$ the collection of all nonempty open subsets of the space $Y$.
A \textit{strategy} of player I in the game $kP(Z,X)$ is a map $\sigma$ defined inductively as follows:
$\sigma(\emptyset)\in \mathcal{K}(X)\times \mathcal{T}_Z$. If the strategy $\sigma$ is defined for the first $n$ moves then an $n$-tuple $(V_0,V_1,\ldots,V_{n-1})\in \mathcal{T}_Z^n$ is called
\textit{admissible} if $K_0\subseteq V_0\subseteq U_0$ and $K_i\subseteq V_i\subseteq U_i$, and $(K_i,U_i)=\sigma(V_0,\ldots V_{i-1})$ for $i\in\{1,\ldots , n-1\}$. For any
admissible $n$-tuple $(V_0,\ldots, V_{n-1})$ we choose a pair $(K_n,U_n)\in \mathcal{K}(V_{n-1})\times \mathcal{T}_{V_{n-1}}$ with $K_n\subseteq U_n$ and we set

$$\sigma(V_0,\ldots ,V_{n-1})=(K_n,U_n).$$

A strategy $\sigma$ of player I in the game $kP(Z,X)$ is called \textit{winning} if player I wins every run of the game $kP(Z,X)$ in which she plays according to the
strategy $\sigma$.

We will need the following simple, though a little technical lemma concerning the game $kP(Z,X)$.

\begin{lemma}\label{technical lemma}
Suppose that $X\subseteq Z$ where $Z$ is compact. Assume that there is
a countable family $\{F_i:i=1,2,\ldots\}$ consisting
of compact subsets of $Z$ and satisfying $\bigcup_{i=1}^\infty F_i\supseteq X$.

If $\sigma$ is a winning strategy of player I in the game $kP(Z,Z\setminus X)$, then for every $k\in \mathbb{N}$ and every admissible tuple $(V_0,\ldots ,V_k)$, there exists $m>k$ and open sets $V_{k+1},\ldots , V_m$ such that the tuple $(V_0,\ldots ,V_m)$ is admissible and $\sigma(V_0,\ldots ,V_m)=(K_{m+1},U_{m+1})$ satisfies $K_{m+1}\cap \bigcup_{i=1}^\infty F_i\neq \emptyset$.
\end{lemma}
\begin{proof}
Striving for a contradiction, suppose that for some admissible $(k+1)$-tuple $(V_0,\ldots ,V_{k})$ we have:
\begin{equation}
    \begin{aligned}
&\mbox{For every } m>k \mbox{, if } (V_0,\ldots,V_k,\ldots, V_m) \mbox{ is admissible and }\\
&\sigma(V_0,\ldots, V_m)=(K_{m+1},U_{m+1})\mbox{, then } K_{m+1}\cap \bigcup \{F_i:i\geq 1\} = \emptyset 
    \end{aligned}
    \tag{$\ast$}\label{games *}
    \end{equation}
We recursively define sets $V_m$, for $m>k$, as follows:
If the sets $V_0,\ldots V_{m-1}$ are already defined in such a way that the tuple $(V_0,\ldots V_{m-1})$ is admissible, we consider the pair
$(K_m,U_m)=\sigma(V_0,\ldots V_{m-1})$. Let
$$V'_m=U_m\cap (Z\setminus F_{(m-k)}).$$
By \eqref{games *}, $K_m\subseteq V'_m$.
Let $V_m$ be an open set in $Z$ satisfying:
$$K_m\subseteq V_m\subseteq \overline{V_m}\subseteq V'_m.$$
It is clear that the tuple $(V_0,\ldots V_{m})$ is admissible and we can proceed with our recursive construction.

In this way we define a play in the game $kP(Z,Z\setminus X)$ in which player I applies her strategy and fails. Indeed, we have $\bigcap_{m=0}^\infty V_m\neq \emptyset$, because $V_{m+1}\subseteq \overline{V_{m+1}}\subseteq V_m\subseteq Z$ and
$Z$ is compact. Moreover, $X\cap \bigcap_{m=0}^\infty V_m=\emptyset$, because $V_{k+i}\cap F_i = \emptyset$ and $X\subseteq \bigcup_{i=1}^\infty F_i$. 
\end{proof}

From the previous lemma we can easily deduce:

\begin{prop}\label{proposition modify strategy}
Suppose that $X\subseteq Z$ where $Z$ is compact. Assume that there is
a countable family $\{F_i:i=1,2,\ldots\}$ consisting
of compact subsets of $Z$ and satisfying $\bigcup_{i=1}^\infty F_i\supseteq X$.

If player I has a winning strategy in the game $kP(Z,Z\setminus X)$, then player I has a winning strategy $\sigma'$ in this game such that every compact set played by player I according to the strategy $\sigma'$, meets $\bigcup_{i=1}^\infty F_i$.
\end{prop}
\begin{proof}
Let $\sigma$ be an arbitrary winning strategy of player I in the game $kP(Z,Z\setminus X)$. We will define a strategy $\sigma'$ recursively. Consider $(K_0,U_0)=\sigma(\emptyset)$. Let $V_0$ be an arbitrary open set in $Z$ such that $K_0\subseteq V_0\subseteq U_0$. By Lemma \ref{technical lemma} there is $m_0$ and sets $V_1,\ldots ,V_{m_0}$ such that the tuple $(V_0,\ldots ,V_{m_0})$ is admissible and if
$\sigma(V_0,\ldots ,V_{m_0})=(K_{m_0},U_{m_0})$, then $K_{m_0}\cap \bigcup_{i=1}^\infty F_i\neq\emptyset$. We define

$$\sigma'(\emptyset)=(K_{m_0},V_{m_0}).$$

If $V^0=V_{m_0+1}$ is an open set in $Z$ with
$K_{m_0}\subseteq V^0\subseteq U_{m_0}$, then the tuple
$(V_0,\dots, V_{m_0},V_{m_0+1})$ is admissible for $\sigma$. Hence, by Lemma \ref{technical lemma}, there is $m_1>m_0+1$ and sets $V_{m_0+2},\ldots ,V_{m_1}$ such that the tuple $(V_0,\ldots ,V_{m_1})$ is admissible and if
$\sigma(V_0,\ldots ,V_{m_1})=(K_{m_1},U_{m_1})$, then $K_{m_1}\cap \bigcup_{i=1}^\infty F_i\neq\emptyset$. We define

$$\sigma'(V^0)=(K_{m_0},V_{m_0})$$

and so on.
\end{proof}

\subsection{Position of a space in its compactification}

It was already observed by Smirnov \cite{Sm} that the Lindel\"of property of a Tychonoff
space $X$ can be conveniently characterized by the way $X$ is placed in its compactification $bX$ of $X$ (cf. \cite[3.12.25]{Eng}).

A similar characterization of the Hurewicz property were obtained by Just \textit{et al.} \cite{JMSS} (for the subsets of the real line),
Banakh and Zdomskyy \cite{BZ} (for separable metrizable spaces) and Tall \cite{Tall} (the general case). We have the following
(see \cite[Theorem 6]{Tall}):

\begin{thrm}\label{Hurewicz characterization}
For any Tychonoff space $X$ the following conditions are equivalent:
\begin{enumerate}
    \item $X$ has the Hurewicz property
    \item For every compactification $b X$ of $X$ and every $\sigma$-compact subset $F$ of the remainder $b X\setminus X$,
    there exists a $G_\delta$-subset $G$ of $b X$ such that $F\subseteq G\subseteq b X\setminus X$.
    \item There exists a compactification $b X$ of $X$ such that for every $\sigma$-compact subset $F$ of the remainder $b X\setminus X$,
    there exists a $G_\delta$-subset $G$ of $b X$ such that $F\subseteq G\subseteq b X\setminus X$.
\end{enumerate}
\end{thrm}

Let $X$ be a space and let $bX$ be a compactification of $X$.
It was proved in \cite{T} that the $k$-Porada game on $bX$ with values in $bX\setminus X$,
characterizes the Menger property of $X$ (cf. Remark \ref{remark Telgarsky}). We have the following:

\begin{thrm}\cite[Theorem 2]{T}\label{thrm_Telgarsky}
If $X\subseteq Z$, where $Z$ is compact
then the following two conditions are equivalent:
\begin{enumerate}
    \item $X$ has the Menger property
    \item Player I has no winning strategy in the $k$-Porada game\\ $kP(Z,Z\setminus X)$
\end{enumerate}
\end{thrm}

\begin{rem}\label{remark Telgarsky}
It is perhaps worth mentioning here that Theorem 2 in \cite{T} asserts actually that the game $kP(Z,Z\setminus X)$ is equivalent to the Menger game.
It is well known however (see e.g. \cite[Theorem 13]{Scheepers}) or \cite[Theorem 2.32]{AD}) that a topological space $X$ is Menger
if and only if player I has no winning strategy in the Menger game.
\end{rem} 

It was recently observed by Krupski and Kucharski \cite{KK} that one can obtain similar characterizations for the projective properties of Hurewicz and Menger.
As usual, by $\beta X$ we denote the \v{C}ech-Stone compactification of $X$. A subset $A$ of a topological space $Z$ is called \textit{zero-set}
if $A=f^{-1}(0)$ for some continuous function $f:Z\to [0,1]$. Vedenissov's lemma (see \cite[1.5.12]{Eng}) asserts that if $Z$ is
a normal space (in particular compact), then $A$ is a zero-set if and only if $A$ is closed $G_\delta$-subset of $Z$. 

The proof of the following assertion is quite easy to derive from Theorem \ref{Hurewicz characterization}.

\begin{prop}\label{proposition proj Hurewicz}
For any Tychonoff space $X$ the following conditions are equivalent:
\begin{enumerate}
    \item $X$ is projectively Hurewicz
    \item For every subset $F$ of $\beta X\setminus X$ being a countable union of zero-sets in $\beta X$,
    there exists a $G_\delta$ subset $G$ of $\beta X$ such that $F\subseteq G\subseteq \beta X\setminus X$.
    \end{enumerate}
\end{prop}

In order to formulate a respective result for the projective Menger property we need the following modification of the $k$-Porada game.
Let $Z$ be a compact space and let $X\subseteq Z$ be a subspace of $Z$.
The \textit{$z$-Porada game} on $Z$ with values in $X$ (denoted by $zP(Z,X)$) is played as $kP(Z,X)$ with the only difference that compact
sets played by player II are required to be additionally zero-sets in $Z$. A strategy for player I in the game $zP(Z,X)$ is defined analogously with
obvious modifications.
We have the following (see \cite{KK}):

\begin{prop}\label{proposition proj Menger}
The following two conditions are equivalent:
\begin{enumerate}
    \item $X$ has the projective Menger property
    \item Player I has no winning strategy in the $z$-Porada game\\ $zP(\beta X,\beta X\setminus X)$
\end{enumerate}
\end{prop}

Let us also note the following simple fact:

\begin{lemma}\label{upgrading compact to zero}
Let $Z$ be a compact space. If $L\subseteq G$ where $L$ is compact and $G$ is a $G_\delta$-subset of $Z$, then there exists a zero-set $L'$ such that
$L\subseteq L'\subseteq G$
\end{lemma}

We omit the obvious proof of the above lemma.

\section{The support map}

Let
$\varphi:C_p(X)\to C_p(Y)$ be continuous and linear.
For $y\in Y$ we define \textit{the support of $y$ with respect to $\varphi$} as the set $\supp_\varphi(y)$ of all $x\in X$ satisfying the
condition that for every neighborhood $U$ of $x$, there is $f\in C_p(X)$ such that
$f[X\setminus U]\subseteq \{0\}$ and $\varphi(f)(y)\neq 0$ (see \cite{A}, \cite[\S 6.8]{vM}).

The following fact is well known (see \cite[Lemma 6.8.1]{vM} and \cite[Lemma 6.8.2]{vM})
\begin{lemma}\label{Lemma_supp}
Let $\varphi:C_p(X)\to C_p(Y)$ be continuous and linear. Then
\begin{enumerate}
    \item $\supp_\varphi(y)$ a finite subset of $X$.
    \item If $f\in C_p(X)$ satisfies $f[\supp_\varphi(y)]\subseteq\{0\}$ then $\varphi(f)(y)=0$.
    \item If $\varphi$ is surjective then $\supp_\varphi(y)\neq\emptyset$ for every $y\in Y$.
    \item The multivalued map $y\mapsto \supp_\varphi(y)$ is lower semi-continuous.
\end{enumerate}
\end{lemma}
Let $\overline{\mathbb{R}}=\mathbb{R}\cup\{\infty\}$ be the one-point compactification of $\RR$ and let $Z$ be a Tychonoff space.
For a function $f\in C_p(Z)$ the function $\widetilde{f}:\beta Z\to \overline{\RR}$ is the continuous extension of $f$ over the \v{C}ech-Stone
compactification $\beta Z$ of $Z$ (i.e., $\widetilde{f}$ is continuous and $\widetilde{f}\upharpoonright Z=f$).
Since the addition is not defined for all pairs of points in $\overline{\mathbb{R}}$, the sum of two functions
$\widetilde{f}$ and $\widetilde{g}$ may not be well defined. However, we have the following:

\begin{lemma}\label{f+g}
 Let $r_1,\ldots ,r_n\in C_p(Z)$ be a finite collection of continuous functions and let $z\in \beta Z$. Let $r=r_1+\dotsb +r_n$. If for every $i\leq n$,
 $\widetilde{r_i}(z)\in \mathbb{R}$, then $\widetilde{r}(z)=\widetilde{r_1}(z)+\dotsb +\widetilde{r_n}(z)$.
\end{lemma}
\begin{proof}
 For $i\leq n$ define $W_i=\{x\in \beta Z:|\widetilde{r_i}(x)-\widetilde{r_i}(z)|<1\}$. Note that this set is well defined because
 $\widetilde{r_i}(z)\in \mathbb{R}$. It is
 also open by continuity of $\widetilde{r_i}$.
 The set $W=\bigcap_{i=1}^n W_i$ is an open neighborhood of $z$ in $\beta Z$, and for every $x\in W$ the quantity $\widetilde{r_1}(x)+\dotsb +\widetilde{r_n}(x)$
 is a well-defined real number. Thus, $(\widetilde{r_1}+\dotsb +\widetilde{r_n})\upharpoonright W$ is a well-defined continuous function on $W$. Since
 $r(x)=r_1(x)+\dotsb +r_n(x)$ for $x\in Z$ and $Z$ is dense in $W$, we must have $\widetilde{r}(x)=\widetilde{r_1}(x)+\dotsb +\widetilde{r_n}(x)$, for $x\in W$.
 In particular, $\widetilde{r}(z)=\widetilde{r_1}(z)+\dotsb +\widetilde{r_n}(z)$.
\end{proof}

It will be convenient to introduce the following definition:
\begin{defin}
Let $\varphi: C_p(X)\to C_p(Y)$ be a linear continuous map and let $y\in Y$.
An open set $U\subseteq \beta X$ is called \textit{$y$-effective} if every function $f\in C_p(X)$ such that $f[X\setminus U]\subseteq \{0\}$ satisfies $\widetilde{\varphi(f)}(y)=0$.
An open set $U\subseteq \beta X$ is called \textit{$y$-ineffective} if it is not $y$-effective.
\end{defin}

For $y\in \beta Y$ we set
$$s_{\varphi}(y)=\{x\in \beta X: \text{every open neighborhood of $x$ is $y$-ineffective}\}.$$

\begin{xrem}
We should point out that the idea of considering the set $s_\varphi(y)$ is not new. The same concept
(for spaces of bounded continuous functions) was used e.g., by Valov in \cite{V1} and \cite{V2}.
\end{xrem}

Directly from the definition we get:
\begin{lemma}\label{lemma_closed}
The set $s_{\varphi}(y)$ is closed in $\beta X$, hence it is compact.
\end{lemma}

\begin{lemma}\label{lemma_s=supp}
If $y\in Y$ then $s_{\varphi}(y)=\supp_\varphi(y)$.
\end{lemma}
\begin{proof}
The inclusion $\supseteq$ is clear. Suppose that there is $x\in s_{\varphi}(y)\setminus \supp_\varphi(y)$.
Since $\supp_\varphi(y)$ is finite (see Lemma \ref{Lemma_supp}), there is an open neighborhood $U$ of $x$ in $\beta X$ such that $U\cap \supp_\varphi(y)=\emptyset$.
Let $f\in C_p(X)$ be such that $f[X\setminus U]\subseteq\{0\}$. Then $f[\supp_\varphi (y)]\subseteq \{0\}$ and hence $\varphi(f)(y)=0$,
by Lemma \ref{Lemma_supp}. This means that $U\ni x$ is $y$-effective, contradicting $x\in s_{\varphi}(y)$.
\end{proof}
\begin{lemma}\label{lemma_effective}
Let $\varphi:C_p(X)\to C_p(Y)$ be a linear continuous map. Let $y\in \beta Y$ and
let $U$ be an open set in $\beta X$ such that $s_{\varphi}(y)\subseteq U$.
If $f\in C_p(X)$ satisfies $f[U\cap X]\subseteq\{0\}$ then $\widetilde{\varphi(f)}(y)=0$.
\end{lemma}
\begin{proof}
Fix $f\in C_p(X)$ with $f[U\cap X]\subseteq\{0\}$. We have $s_{\varphi}(y)\subseteq U$, so if $x\in\beta X\setminus U$ then there exists an open neighborhood $U_x$ of $x$ in $\beta X$ which is $y$-effective. For each $x\in\beta X\setminus U$ let $V_x$ be an open neighborhood of $x$ in $\beta X$ satisfying $V_x\subseteq \overline{V_x}\subseteq U_x$. The family $\{V_x:x\in \beta X\setminus U\}$ covers the compact set $\beta X\setminus U$. Let $\{V_{x_1},\ldots ,V_{x_n}\}$ be its finite subcover. Let
$$F=\overline{V_{x_1}}\cup\dotsb \cup\overline{V_{x_n}}.$$

For $i=1,\ldots , n$, let $g_i:\beta X\to [0,1]$ be a continuous function that satisfies
    $$g_i[\overline{V_{x_i}}]=\{1\}\quad \text{and}\quad g_i[\beta X\setminus U_{x_i}]=\{0\}.$$

For $i=1,\ldots , n$ there exists a function $f_i\in C_p(\beta X)$ such that
\begin{align*}
f_i(x)=
\begin{cases}
\frac{g_i(x)}{g_1(x)+\dotsb +g_n(x)}&\quad \text{for}\;x\in F\\
0&\quad \text{for}\;x\in\beta X\setminus (U_{x_1}\cup\dotsb \cup U_{x_n})
\end{cases}
\end{align*}
Let $h_i=f\cdot(f_i\upharpoonright X)$ be the product of the functions $f$ and $f_i\upharpoonright X$.
The function $h_i$ has the following property:
\begin{equation}\label{condition*}
\text{If }\; x\in X\setminus U_{x_i}\; \text{ then }\; h_i(x)=0. \tag{$\ast$}
\end{equation}
Indeed, if $x\in F\setminus U_{x_i}$ then $f_i(x)=g_i(x)/(g_1(x)+\dotsb g_n(x))=0$, because $g_i(x)=0$ for $x\notin U_{x_i}$. Hence $h_i(x)=f(x)\cdot f_i(x)=0$. On the other hand, if $x\notin F$ then $x\in U$ so $f(x)=0$. Hence $h_i(x)=f(x)\cdot f_i(x)=0$ too.

Each set $U_{x_i}$ is $y$-effective, thus $\widetilde{\varphi(h_i)}(y)=0$, for every $i=1,\ldots ,n$, by \eqref{condition*}.
Let $h=h_1+\dotsb +h_n$. We can apply Lemma \ref{f+g} with $r=\varphi(h)$ and $r_i=\varphi(h_i)$, obtaining
$$\widetilde{\varphi(h)}(y)=0.$$

We claim that $h=f$. Indeed, if $x\notin F$ then $x\in U$. Thus
$f(x)=0$, by our assumption on $f$. It follows that for such $x$ and for all $i\in \{1,\ldots ,n\}$ we have $h_i(x)=f(x)\cdot f_i(x)=0$. Thus, $h(x)=h_1(x)+\dotsb +h_n(x)=0=f(x)$ for $x\notin F$.
Now suppose that $x\in F$. We have
\begin{align*}
h(x)&=h_1(x)+\dotsb +h_n(x)=f(x)\cdot (f_1(x)+\dotsb + f_n(x))\\
&=f(x)\cdot \sum_{i=1}^n\frac{g_i(x)}{g_1(x)+\dotsb +g_n(x)}=f(x).
\end{align*}
\end{proof}
\begin{cor}\label{cor_nonempty}
If $\varphi:C_p(X)\to C_p(Y)$ is a linear continuous surjection, then
for every $y\in\beta Y $ the set $s_{\varphi}(y)$ is nonempty.
\end{cor}
\begin{proof}
Take $y\in \beta Y$ and suppose that $s_{\varphi}(y)=\emptyset$.
Then $\emptyset$ is an open set containing $s_{\varphi}(y)$.  Let $g\in C_p(Y)$ be such that $\widetilde{g}(y)=1$. Since $\varphi$ is onto, there is $f\in C_p(X)$ with $\varphi(f)=g$. Clearly $f[\emptyset]\subseteq \{0\}$ so
$$\widetilde{g}(y)=\widetilde{\varphi(f)}(y)=0,$$
by Lemma \ref{lemma_effective}, which is a contradiction.
\end{proof}
\begin{prop}\label{s is lsc}
Let $\varphi:C_p(X)\to C_p(Y)$ be a continuous surjection. The set-valued map
$s_\varphi:\beta Y\to \mathcal{K}(\beta X)$, given by the assignment $y\mapsto s_{\varphi}(y)$ is
lower semi-continuous.
\end{prop}
\begin{proof}
By Lemma \ref{lemma_closed} and Corollary \ref{cor_nonempty}, the map $s$ is well defined. Let $U\subseteq \beta X$ be open. We need to show that the set $$s_\varphi^{-1}(U)=\{y\in \beta Y:s_{\varphi}(y)\cap U\neq\emptyset\}$$ is open in $\beta Y$. Pick $y_0\in s_\varphi^{-1}(U)$ and take $x_0\in s_{\varphi}(y_0)\cap U$ witnessing $s_{\varphi}(y_0)\cap U\neq \emptyset$. Let $V$ be an open neighborhood of $x_0$ such that $\overline{V}\subseteq U$. Since $x_0\in s_{\varphi}(y_0)$, the set $V\ni x_0$ is $y_0$-ineffective. Therefore, there is $f\in C_p(X)$ such that $f[X\setminus V]\subseteq\{0\}$ and
$\widetilde{\varphi(f)}(y_0)\neq 0$. Consider the open set
$$W=\{y\in \beta Y:\widetilde{\varphi(f)}(y)\neq 0\}.$$
Clearly, $y_0\in W$. We claim that $W\subseteq s_\varphi^{-1}(U)$. Take $y\in W$. If $s_{\varphi}(y)\cap U =\emptyset$ then $s_{\varphi}(y)\subseteq \beta X\setminus \overline{V}$, the set $\beta X\setminus \overline{V}$ is open in $\beta X$, and $f[X\setminus \overline{V}] \subseteq\{0\}$. Hence $\widetilde{\varphi(f)}(y)=0$, by
Lemma \ref{lemma_effective}. A contradiction with $y\in W$.

It follows that $y_0\in W\subseteq s_\varphi^{-1}(U)$, where $W$ is open. Since $y_0$ was chosen arbitrarily, the set $s_\varphi^{-1}(U)$ is open.
\end{proof}

\begin{cor}\label{Y_n}
For every integer $n\geq 1$, the set $\widetilde{Y_n}=\{y\in \beta Y :|s_{\varphi}(y)|\leq n\}$ is closed in $\beta Y$.
\end{cor}
\begin{proof}
Let $y\in \beta Y\setminus \widetilde{Y_n}.$ Then $s_{\varphi}(y)$ has at least $n+1$ elements, so there are distinct
$x_1,\ldots ,x_{n+1}\subseteq s_{\varphi}(y)$. Let $V_1,\ldots, V_{n+1}$ be pairwise disjoint open subsets of $\beta X$ such that $x_i\in V_i$,
for $i=1,\ldots , n+1$. By lower semi-continuity of $s_\varphi$, cf. Proposition \ref{s is lsc}, the set $W=\bigcap_{i=1}^{n+1} s_\varphi^{-1}(V_i)$
is open and clearly $y\in W$. For any $z\in W$, the set $s_{\varphi}(z)$ meets $n+1$ pairwise disjoint sets $V_i$.
Hence $y\in W\subseteq \beta Y\setminus \widetilde{Y_n}$.
\end{proof}

Similarly, from lower semi-continuity of the support map $y\mapsto \supp_\varphi(y)$ (see Lemma \ref{Lemma_supp}) it follows that the set
$$Y_n=\{y\in Y:|\supp_\varphi(y)|\leq n\}$$
is closed in $Y$. Also, since $s_{\varphi}(y)=\supp_\varphi(y)$ for $y\in Y$ (cf. Lemma \ref{lemma_s=supp}), we have $Y_n\subseteq \widetilde{Y_n}$.

\section{Technical Lemmata}
In this section we will use some ideas from Okunev \cite{Ok} (see also \cite{K}).
Let $Z$ be a Tychonoff space. For $\varepsilon>0$ and a finite set $F=\{z_1,\ldots , z_k\}\subseteq \beta Z$ we set
$$O_Z(F,\varepsilon)=\{f\in C_p(Z):|\widetilde{f}(z_i)|<\varepsilon, \;i=1,\ldots ,k\}.$$
For a point $z\in \beta Z$ and $\varepsilon>0$, let
$$\bar{O}_Z(z,\varepsilon)=\{f\in C_p(Z):|\widetilde{f}(z)|\leq \varepsilon\}.$$
Note that for $z\in Z$ the set $\bar{O}_Z(z,\varepsilon)$ is closed in $C_p(Z)$, whereas for $z\in \beta Z\setminus Z$ it is dense and has empty interior in
$C_p(Z)$.
 
Let $\varphi:C_p(X)\to C_p(Y)$ be a linear homeomorphism. By linearity
$\varphi(\underline{0})=\underline{0}$, where $\underline{0}$ is the constant function equal to 0 in the respective space.

For positive integers $k$ and $m$ we define the set
$$Z_{m,k}=\left\{(y,F)\in Y\times [X]^{\leq k}:\varphi\left(O_X\left(F,\tfrac{1}{m}\right)\right)\subseteq \bar{O}_Y(y,1)  \right\}.$$
We should remark that in \cite{Ok} sets $Z_{m,k}$ are defined in a slightly different way,
i.e., the product $X^k$ is used instead of the hyperspace $[X]^{\leq k}$. However our (cosmetic) change does not affect the arguments from \cite{Ok}.

Now, for positive integers $k$ and $m$, let $S_{m,k}$ be the closure of $Z_{m,k}$ in the (compact) space $\beta Y\times [\beta X]^{\leq k}$. Recall that $[\beta X]^{\leq k}$ is endowed with the Vietoris topology (cf. Section 2.1).
We have, cf. \cite[Lemma 1.4]{Ok} (we reproduce the proof here for the convenience of the reader):
\begin{lemma}\label{Okunev_1.4}
If $(y,F)\in S_{m,k}$, then $\varphi\left(O_X\left(F,\tfrac{1}{m}\right)\right)\subseteq \bar{O}_Y(y,1)$.
\end{lemma}
\begin{proof}
Otherwise, there is $f\in C_p(X)$ with $|\widetilde{f}(x)|<\tfrac{1}{m}$ for each $x\in F$ and $|\widetilde{\varphi(f)}(y)|>1$.
The set
$$U=\left\{A\in [\beta X]^{\leq k}: \widetilde{f}(A)\subseteq \left( -\tfrac{1}{m},\tfrac{1}{m} \right)\right\}$$
is open in $[\beta X]^{\leq k}$ and $F\in U$.
Similarly, the set
$$V=\left\{z\in \beta Y:|\widetilde{\varphi(f)}(z)|>1\right\}$$
is an open neighborhood of $y$ in $\beta Y$.
Since $(y,F)\in S_{m,k}$, the open set $V\times U$ has a nonempty intersection with $Z_{m,k}$.
This, however contradicts the definition of $Z_{m,k}$.
\end{proof}

For $k,m\geq 1$, define
$$C_{m,k}=\pi_{\beta Y}(S_{m,k}),$$
where $\pi_{\beta Y}:\beta Y\times [\beta X]^{\leq k}\to\beta Y$ is the projection onto the first factor. Clearly, $C_{m,k}$ is closed in $\beta Y$.

Recall that $Y_n=\{y\in Y:|\supp_\varphi(y)|\leq n\}$. We set $$A_{m,n}=Y_n\cap C_{m,n}.$$

\begin{lemma}\label{lemat_supp_in_Z_mn}
If $y\in Y_{n}$, then for some $m\geq 1$,
$(y,\supp_\varphi(y))\in Z_{m,n}$ and thus $y\in A_{m,n}$.
\end{lemma}
\begin{proof}
Take $y\in Y_n$. By continuity of $\varphi$ there is a finite set
$F=\{x_1,\ldots , x_k\}\subseteq X$ and $m\geq 1$ with
$$O_X\left(F,\tfrac{1}{m}\right)\subseteq \varphi^{-1}\left(\bar{O}_Y(y,1)\right).$$
We will check that the number $m$ does the job.
To this end,
consider a function $f\in C_p(X)$ satisfying 
$$|f(x)|<\tfrac{1}{m},\quad\text{for every}\;x\in \supp_\varphi(y).$$

Striving for a contradiction, suppose that $|\varphi(f)(y)|>1$.  and let $g\in C_p(X)$ be such that $g\upharpoonright\supp_\varphi(y)=f\upharpoonright\supp_\varphi(y)$ and $g(x)=0$, for every $x\in F\setminus\supp_\varphi(y)$. Since $g$ and $f$ agree on $\supp_\varphi(y)$, we have
$\varphi(g)(y)=\varphi(f)(y)>1$ (see Lemma \ref{Lemma_supp}). On the other hand,
$$g\in O_X\left(F,\tfrac{1}{m}\right)\subseteq \varphi^{-1}\left(\bar{O}_Y(y,1)\right),$$
a contradiction.
\end{proof}

\begin{prop}\label{proposition A_nn}
$Y=\bigcup_{n=1}^\infty A_{n,n}$ and $A_{n,n}\subseteq A_{m,m}$ for $m\geq n$.
\end{prop}
\begin{proof}
Let $y\in Y$.
By Lemma \ref{Lemma_supp}, $Y=\bigcup_{n=1}^\infty Y_n$ so $y\in Y_n$, for some $n\geq 1$.
From Lemma \ref{lemat_supp_in_Z_mn}, we infer that $y\in C_{m,n}$ for some $m$. Note that if $m\leq k$, then $C_{m,n} \subseteq C_{k,n}$, so we can assume that $m>n$, for otherwise $y\in C_{n,n}$ and we are done.  
Since $Y_n\subseteq Y_k$ for $k\geq n$, we have $y\in Y_k$ for all $k\geq n$. In particular $y\in Y_m$.
So $y\in Y_m\cap C_{m,n}$, where $m>n$. But clearly $m>n$ implies $C_{m,n}\subseteq C_{m,m}$, whence $y\in A_{m,m}$.
This gives the equality $Y=\bigcup_{n=1}^\infty A_{n,n}$. The inclusion $A_{n,n}\subseteq A_{m,m}$, for $m\geq n$, is clear.
\end{proof}

Let $B_{m,n}$ be the closure of $A_{m,n}$ in $\beta Y$. Since $Y_n\subseteq \widetilde{{Y}_n}$ and both $\widetilde{{Y}_n}$ and
$C_{m,n}$ are closed in $\beta Y$, we infer that
$$B_{m,n}\subseteq \widetilde{{Y}_n}\cap C_{m,n}.$$
In particular, if $y\in B_{m,n}$, then the set $s_{\varphi}(y)$ is at most $n$-element subset of $\beta X$.
\begin{lemma}\label{s is in S_mn 1}
If $y\in B_{m,n}$, then $\varphi\left(O_X\left(s_{\varphi}(y),\tfrac{1}{m}\right)\right)\subseteq \bar{O}_Y(y,1)$
\end{lemma}
\begin{proof}
Pick $f\in C_p(X)$ such that $|\widetilde{f}(x)|<\tfrac{1}{m}$ for $x\in s_{\varphi}(y)$.
The set $U=\{x\in \beta X: |\widetilde{f}(x)|<\tfrac{1}{m}\}$ is open in $\beta X$ and $s_\varphi(y)\subseteq U$. Let $V$ be an open subset of $\beta X$ satisfying
\begin{equation}
 s_\varphi(y)\subseteq V\subseteq \overline{V}\subseteq U.\label{s in V}
\end{equation}
Since $y\in B_{m,n}\subseteq C_{m,n}$,
there is $F\in [\beta X]^{\leq n}$ with $(y,F)\in S_{m,n}$.
Let $\widetilde{g}\in C_p(\beta X)$ be a function satisfying
$$\widetilde{g}\upharpoonright \overline{V}=\widetilde{f}\upharpoonright \overline{V}\quad\mbox{and}\quad \widetilde{g}(x)=0
\mbox{ for each }x\in F\setminus \overline{V}.$$
Denote by $g$ the function $\widetilde{g}\upharpoonright X$, i.e., the restriction of $\widetilde{g}$ to $X$.
Clearly, $g\in O_X\left(F,\tfrac{1}{m}\right)$, so by Lemma \ref{Okunev_1.4}, we have
\begin{equation}
 |\widetilde{\varphi(g)}(y)|\leq 1 \label{|g|<1}
\end{equation}
Further, $(\widetilde{f}-\widetilde{g})\upharpoonright V=0$. So from \eqref{s in V} and Lemma \ref{lemma_effective} we infer that
\begin{equation}
 \widetilde{\varphi(f-g)}(y)=0. \label{f-g=0}
\end{equation}

By linearity of $\varphi$ we get

$$\varphi(f-g)+\varphi(g)=\varphi(f).$$

Inequality \eqref{|g|<1} and equation \eqref{f-g=0} ensure that Lemma \ref{f+g} can be applied with $r=\varphi(f)$, $r_1=\varphi(f-g)$ and $r_2=\varphi(g)$, whence
$$|\widetilde{\varphi(f)}(y)|=|\widetilde{\varphi(f-g)}(y)+\widetilde{\varphi(g)}(y)|\leq|\widetilde{\varphi(f-g)}(y)|+|\widetilde{\varphi(g)}(y)|\leq 1,$$
by \eqref{|g|<1} and \eqref{f-g=0}.
\end{proof}

From the previous lemma we get the following:

\begin{prop}\label{proposition-BmnAmn}
For every $y\in B_{m,n}\setminus A_{m,n}$
the set $s_{\varphi}(y)\cap (\beta X\setminus X)$ is nonempty.
\end{prop}
\begin{proof}
Let $y\in B_{m,n}\setminus A_{m,n}$.
Since $A_{m,n}$ is closed in $Y$ and $B_{m,n}$ is the closure of $A_{m,n}$ in $\beta Y$, we have $y\in \beta Y\setminus Y$.
So the set $\bar{O}_Y\left(y,\tfrac{1}{m}\right)$ has empty interior in $C_p(Y)$. By Lemma \ref{s is in S_mn 1}, we have
$$\varphi\left(O_X\left(s_{\varphi}(y),\tfrac{1}{m}\right)\right)\subseteq \bar{O}_Y(y,1).$$
Now, if $s_{\varphi}(y)$ were a subset of $X$, then the set $\varphi\left(O_X\left(s_{\varphi}(y),\tfrac{1}{m}\right)\right)$ would be open,
contradicting emptiness of the interior of $\bar{O}_Y\left(y,1\right)$.
\end{proof}

\begin{prop}\label{proposition-y in s(s(y))}
Let $\varphi:C_p(X)\to C_p(Y)$ be a linear homeomorphism.
If $y\in B_{m,n}$, then there exists $x\in s_{\varphi}(y)$ with $y\in s_{\varphi^{-1}}(x)$.
\end{prop}
\begin{proof}
Since $B_{m,n}\subseteq \widetilde{Y}_n$, the set $s_{\varphi}(y)$ is at most $n$-element. Thus $K=\bigcup\{s_{\varphi^{-1}}(x):x\in s_{\varphi}(y)\}$ is compact,
being a finite union of compact sets $s_{\varphi^{-1}}(x)$.

Striving for a contradiction, suppose that $y\notin K$. Let $U$ be an open set in $\beta Y$ with
$K\subseteq U$ and $y\notin U$. Let $V$ be an open set in $\beta Y$ with
$$K\subseteq V\subseteq \overline{V}\subseteq U.$$

Let $f\in C_p(\beta Y)$ satisfies $f(\overline{V})\subseteq \{0\}$ and $f(y)=2$. From Lemma \ref{lemma_effective} (applied to the map $\varphi^{-1}$), we get
$$\widetilde{\varphi^{-1}(f\upharpoonright Y)}(x)=0,\quad\text{for every}\;x\in s_{\varphi}(y).$$
Combining this with Lemma \ref{s is in S_mn 1}, we get
$$|\widetilde{\varphi(\varphi^{-1}(f\upharpoonright Y))}(y)|=|f(y)|\leq 1,$$
which contradicts $f(y)=2$.
\end{proof}

\section{The main results}

Let $M$ be a separable metrizable space and let $h:Y\to M$ be a continuous surjection.
Since $M$ is separable metrizable it has a metrizable compactification $bM$. Let $\widetilde{h}:\beta Y\to bM$ be a continuous extension of
$h$.
Denote by $d$
a metric on $bM$ that generates the topology of $bM$.

For a natural number $k\geq 1$ we define sets

\begin{align*}
E_k&=\{y\in Y: \left(\forall a,b\in h(s_{\varphi^{-1}}(s_{\varphi}(y)))\right)\; a\neq b \Rightarrow d(a,b)\geq \tfrac{1}{k}\}\\
F_k&=\{y\in \beta Y: \left(\forall a,b\in \widetilde{h}(s_{\varphi^{-1}}(s_{\varphi}(y)))\right)\; a\neq b \Rightarrow d(a,b)\geq \tfrac{1}{k}\}
\end{align*}
It is easy to prove the following
\begin{lemma}\label{lemma-properties of E_k}
The sets $E_k$ and $F_k$ have the following properties:
\begin{enumerate}[(i)]
    \item The set $E_k$ is closed in $Y$, for every $k\geq 1$.
    \item The set $F_k$ is closed in $\beta Y$, for every $k\geq 1$.
    \item $\bigcup_{k=1}^\infty E_k=Y$.
    \item If $k\leq r$, then $E_k\subseteq E_r$.
\end{enumerate}
\end{lemma}
\begin{proof}
Let $y\in Y\setminus E_k$. By Lemma \ref{lemma_s=supp},  $s_\varphi(y)=\supp_{\varphi}(y)$, so $y\notin E_k$ means that
there are distinct $a,b\in h(\supp_{\varphi^{-1}}(\supp_{\varphi}(y)))$ with $d(a,b)<\tfrac{1}{k}$. Let $\varepsilon >0$ be such that
\begin{equation}
  \varepsilon<\frac{1}{2k}-\frac{d(a,b)}{2} \quad \mbox{and} \quad \varepsilon<\frac{d(a,b)}{2}  \label{epsilon}
 \end{equation}
For $x\in \{a,b\}$, let $B_x$ be an $\varepsilon$-ball in the space $M$, centered at $x$.

The set $$V_x=h^{-1}(B_x)$$ is open in  $Y$
and
\begin{equation}\label{4}
 V_x\cap \supp_{\varphi^{-1}}(\supp_{\varphi}(y))\neq\emptyset,   
\end{equation}
for $x\in\{a,b\}$. Since the map $\supp_{\varphi^{-1}}$ is lower semi-continuous, the set $$W_x=\supp_{\varphi^{-1}}^{-1}(V_x)$$ is open in $X$ and, by \eqref{4},
$$\supp_{\varphi}(y)\cap W_x\neq \emptyset.$$
It follows that, for $x\in\{a,b\}$, the set
$$U_x=\supp_\varphi^{-1}(W_x)$$
is an open neighborhood of $y$ in $Y$. Put
$$U=U_a\cap U_b.$$

If $z\in U$ then $$V_x\cap\supp_{\varphi^{-1}}(\supp_{\varphi}(z))\neq\emptyset,$$
and hence, there is
$$\xi_x\in B_x\cap h(\supp_{\varphi^{-1}}(\supp_{\varphi}(z))),$$
for $x\in\{a,b\}$. By \eqref{epsilon} the balls $B_a$ and $B_b$ are disjoint, so $\xi_a\neq \xi_b$ and
$$d(\xi_a, \xi_b)<2\varepsilon+d(a,b)<\frac{1}{k},$$
by \eqref{epsilon}. This shows that $U\cap E_k=\emptyset$ and finishes the proof of (i). The proof of (ii) is analogous.

Assertion (iii) follows from the fact that for $y\in Y$ we have
$s_\varphi(y)=\supp_\varphi(y)\subseteq X$ (Lemma \ref{lemma_s=supp}) and thus the set $s_{\varphi^{-1}}(s_{\varphi}(y))=\supp_{\varphi^{-1}}(\supp_\varphi(y))$ is finite, by
Lemma \ref{Lemma_supp}.

Assertion (iv) is clear.
\end{proof}

Now, for $n\geq 1$, let
$$H_n=A_{n,n}\cap E_n$$
and let $\overline{H_n}$ be the closure of $H_n$ in $\beta Y$. The sets $H_n$ and $\overline{H_n}$ have the following properties:
\begin{obs}\label{H1H2}
 $\bigcup_{n=1}^\infty H_n=Y$ and $H_n\subseteq H_{n+1}$ for all $n\geq 1$.
\end{obs}
\begin{proof}
 This follows immediately from Proposition \ref{proposition A_nn} and Lemma \ref{lemma-properties of E_k} (iii)--(iv).
\end{proof}

\begin{obs}\label{H3}
 If $y\in  \overline{H_n}\setminus H_n$, then $s_\varphi(y)\cap (\beta X\setminus X)\neq\emptyset$.
\end{obs}
\begin{proof}
Since $H_n$ is closed in $Y$,
we have
$\overline{H_n}\setminus H_n\subseteq B_{n,n}\setminus A_{n,n}$. So it is enough to apply Proposition \ref{proposition-BmnAmn}.
\end{proof}

\begin{obs}\label{H4}
 For every $y\in \overline{H_n}$ we have $y\in s_{\varphi^{-1}}(s_\varphi(y))$.
\end{obs}
\begin{proof}
According to Lemma \ref{lemma-properties of E_k}, we have
$$\overline{H_n}=\overline{A_{n,n}\cap E_n}\subseteq B_{n,n}\cap F_n.$$
Hence, our assertion follows from Proposition \ref{proposition-y in s(s(y))}.
\end{proof}
\begin{obs}\label{H5}
For every $y\in \overline{H_n}$ and for all distinct $a,b\in \widetilde{h}(s_{\varphi^{-1}}(s_{\varphi}(y)))$ we have $d(a,b)\geq \tfrac{1}{n}$.
\end{obs}
\begin{proof}
 Again, by Lemma \ref{lemma-properties of E_k} we have  $\overline{H_n}\subseteq B_{n,n}\cap F_n.$ So the assertion follows from the definition of the set $F_n$.
\end{proof}

For each $n\geq 1$, we define a set-valued mapping $e_n:\overline{H_n}\to \mathcal{K}(\beta X)$ by the formula
$$e_n(y)=\{x\in s_{\varphi}(y):\widetilde{h}(y)\in \widetilde{h}(s_{\varphi^{-1}}(x))\}.$$
Note that the set $e_n(y)$ is finite because the set $s_\varphi(y)$ is finite for $y\in \overline{H_n}$. Also, $e_n(y)$ is nonempty, by Observation \ref{H4}. So the map
$e_n$ is well defined.
\begin{lemma}\label{lemma e_n is lsc}
For every $n\geq 1$, the map $e_n: \overline{H_n}\to \mathcal{K}(\beta X)$ is lower semi-continuous.
\end{lemma}
\begin{proof}
Take an open set $U\subseteq \beta X$. Pick $y\in e_n^{-1}(U)$ and take $x_0\in e_n(y)\cap U$ witnessing $e_n(y)\cap U\neq \emptyset$.

We need to show that there is an open set $W$ in $\beta Y$ with
$$y\in W\cap \overline{H_n}\subseteq e_n^{-1}(U).$$

Denote by $B$ the ball in $bM$ of radius $\tfrac{1}{2n}$
centered at $\widetilde{h}(y)$ .
Since the map $s_{\varphi^{-1}}:\beta X\to \mathcal{K}(\beta Y)$ is lower semi-continuous (see Proposition \ref{s is lsc}), the set
$$V=\{x\in \beta X: s_{\varphi^{-1}}(x)\cap \widetilde{h}^{-1}(B)\neq \emptyset\}\cap U$$
is open in $\beta X$ and $x_0\in V$.

We set

$$W=\{z\in \beta Y:s_\varphi(z)\cap V\neq\emptyset\}\cap \widetilde{h}^{-1}(B).$$

Note that $W$ is open in $\beta Y$ (by Proposition \ref{s is lsc}) and $y\in W$ because $x_0\in V\cap e_n(y)\subseteq V\cap s_\varphi(y)$.

We claim that $W$ is as required. Indeed, pick $z\in W\cap \overline{H_n}$. We have
\begin{align}
    &\widetilde{h}(z)\in B \quad\quad \mbox{ and}  \label{*} \\
    &s_\varphi(z)\cap V\neq \emptyset \label{**}
\end{align}

Let $x_1$ be a witness for \eqref{**}, i.e.,
\begin{align}
&x_1\in s_\varphi(z)\cap U  \quad\quad \mbox{ and} \label{+} \\
&s_{\varphi^{-1}}(x_1)\cap \widetilde{h}^{-1}(B)\neq \emptyset.\label{++}
\end{align}

By \eqref{++}, there is $z'\in s_{\varphi^{-1}}(x_1)$ such that
\begin{equation}
\widetilde{h}(z')\in B.  \label{+++}    
\end{equation}
On the other hand, since $z\in \overline{H_n}$, we infer from Observation \ref{H4}, that there is $x_2\in s_\varphi(z)$ (possibly $x_2=x_1$) such that
$z\in s_{\varphi^{-1}}(x_2)$. By \eqref{*} and \eqref{+++}  we must have
\begin{equation}
\widetilde{h}(z)=\widetilde{h}(z'). \label{++++}
\end{equation}
For otherwise
$a=\widetilde{h}(z)$ and  $b=\widetilde{h}(z')$ would be distinct elements of $\widetilde{h}(s_{\varphi^{-1}}(s_\varphi(z))$ satisfying
$$d(a,b)\leq d\left(a,\widetilde{h}(y)\right)+d\left(\widetilde{h}(y),b\right)<1/2n+1/2n=1/n,$$
by definition of $B$. However, this would contradict $z\in \overline{H_n}$, by Observation \ref{H5}.

Now, \eqref{++++} gives $\widetilde{h}(z)\in\widetilde{h}(s_{\varphi^{-1}}(x_1))$. But this means that $x_1\in e_n(z)$ and thus by \eqref{+},
$x_1\in e_n(z)\cap U.$ In particular the latter set is nonempty.
\end{proof}

\begin{rem}
Clearly the sets $H_n$ and $\overline{H_n}$ and the map $e_n$ depend on the function $h:Y\to M$. In what follows we will
always be given a function $h:Y\to M$. The sets $H_n$, $\overline{H_n}$ and the map $e_n$ will be associated to the given function $h$.
\end{rem}

It will be convenient to use the following notation. For a continuous map $f:S\to T$ between topological spaces $S$ and $T$ and a set $A\subseteq S$ we denote by $f^\#(A)$ the set
$T\setminus f(S\setminus A)$. It is straightforward to verify the following:

\begin{prop}\label{proposition hash}
Suppose that $f:S\to T$ is a continuous map between topological spaces $S$ and $T$. Then:
\begin{enumerate}[(a)]
\item If $S$ is compact and $U\subseteq S$ is open, then $f^\#(U)$ is open in $T$.
\item For any $t\in T$ and $A\subseteq S$, if $t\in f^\#(A)$, then $f^{-1}(t)\subseteq A$.
\item For any $A \subseteq S$ and $B\subseteq T$, if $f^{-1}(B)\subseteq A$, then $B\subseteq f^\#(A)$
\end{enumerate}
\end{prop}

We are ready now to present proofs of the results announced in the Introduction.

\begin{proof}[Proof of Theorem \ref{thrm proj Hurewicz}]
By symmetry, it is enough to show that the projective Hurewicz property of $X$ implies the projective Hurewicz property of $Y$. Suppose that $X$ is projectively
Hurewicz and
fix a continuous surjection
$h:Y\to M$ that maps $Y$ onto a separable metrizable space $M$.
Let $bM$ be a metrizable compactification of $M$ and let $\widetilde{h}:\beta Y\to bM$ be a continuous extension of $h$.
Denote by $d$ a metric on $bM$ that generates the topology of $bM$. Note that

\begin{equation}
    \mbox{If }A\subseteq bM\setminus M, \mbox{ then }  \widetilde{h}^{-1}(A)\subseteq \beta Y\setminus Y. \label{9'}
\end{equation}

In order to prove that $M$ is Hurewicz, we will employ 
Theorem \ref{Hurewicz characterization}. For this purpose, take a $\sigma$-compact set $F\subseteq bM\setminus M$.
Write $F=\bigcup_{i=1}^\infty K_i$, where each $K_i$ is compact and $K_i\subseteq K_{i+1}$.
We need to show that there is a $G_\delta$-subset $G$ of $bM$ with $F\subseteq G\subseteq bM\setminus M$.

If no $K_i$ intersects $\bigcup_{n=1}^\infty \widetilde{h}(\overline{H_n})$, then we are done because the complement of the latter union in $bM$ is a
$G_\delta$-subset of $bM\setminus M$ (by Observation \ref{H1H2} and surjectivity of $h:Y\to M$). So suppose that, for some $i$, the set $K_i$ meets
$\bigcup_{n=1}^\infty \widetilde{h}(\overline{H_n})$ and let $i_0$ be the first such $i$.
Since the family $\{K_i:i=1,2,\ldots\}$ is increasing, $K_i$ meets $\bigcup_{n=1}^\infty \widetilde{h}(\overline{H_n})$, for every $i\geq i_0$.
In order to find the required $G_\delta$-set $G$ it suffices to find such set for the family $\{K_i:i\geq i_0\}$, i.e., it is
enough to find a $G_\delta $-subset $G'$ of $bM$ such that $\bigcup_{i=i_0}^\infty K_i\subseteq G'\subseteq bM\setminus M$.
This is because the set $\bigcup_{i=1}^{i_0-1} K_i$ is contained in a $G_\delta$-subset of $bM\setminus M$
(the complement of $\bigcup_{n=1}^\infty \widetilde{h}(\overline{H_n})$ in $bM$) and the union of two $G_\delta$-sets is $G_\delta$.

For each $i\geq i_0$ there is a positive integer $n_i$ such that the compact set
$$K'_{n_i}=\widetilde{h}^{-1}(K_i)\cap \overline{H_{n_i}}$$
is nonempty. By Observation \ref{H1H2} we can additionally require that $n_{i_0}<n_{i_0+1}<\dots$.

Let $i\geq i_0$. Since $bM$ is metrizable, the set $K_i$ is $G_\delta$ in $bM$.
It follows from Proposition \ref{s is lsc} and Lemma \ref{G_delta via lsc map} that the set
$$G_i=s^{-1}_{\varphi^{-1}}\left( \widetilde{h}^{-1}(K_i)\right)=\{x\in \beta X:s_{\varphi^{-1}}(x)\cap\widetilde{h}^{-1}(K_i)\neq\emptyset\}$$
is a $G_\delta$-set in $\beta X$. In addition, by \eqref{9'}, Lemma \ref{lemma_s=supp} and
Lemma \ref{Lemma_supp}, we have $G_i\subseteq \beta X\setminus X$.

The map $e_{n_i}$ restricted to $K'_{n_i}$ is lower semi-continuous (by Lemma \ref{lemma e_n is lsc}) and note that if $y\in K'_{n_i}$
then $e_{n_i}(y)\subseteq G_i\subseteq \beta X\setminus X$ (by definition of $e_{n_i}$).
Thus, we may consider the map $e_{n_i}\upharpoonright K'_{n_i}$ as a (lower semi-continuous) map into $\mathcal{K}(G_i)$.
By Theorem \ref{theorem_Bouziad}, this map admits a compact section, i.e., there is a compact set $L_i\subseteq G_i$ such that
$$L_i\cap e_{n_i}(y)\neq \emptyset,\mbox{ for every } y\in K'_{n_i}.$$
In particular, since $e_{n_i}(y)\subseteq s_\varphi(y)$, we have

\begin{equation}
L_i\cap s_{\varphi}(y)\neq \emptyset,\mbox{ for every } y\in K'_{n_i}.\label{10'}
\end{equation}

Using Lemma \ref{upgrading compact to zero}, we can enlarge the set $L_i$ to a zero-set in $\beta X$ contained in $G_i$.
Clearly, this is still a section of $e_{n_i}$, so without loss
of generality we can assume that each $L_i$ is a zero-set.

The space $X$ is projectively Hurewicz and $L=\bigcup_{i=i_0}^\infty L_i\subseteq \beta X\setminus X$, where all $L_i$'s are zero-sets in $\beta X$.
Hence, by Proposition \ref{proposition proj Hurewicz},
there is a $G_\delta$-set $P$ in $\beta X$ with
\begin{equation}
L\subseteq P\subseteq \beta X\setminus X. \label{14''}
\end{equation}
We can write
$$P=\bigcap_{i=i_0}^{\infty} P_i,$$
where the sets $P_i$ are open in $\beta X$ and form a
decreasing sequence, i.e., $P_i\supseteq P_{i+1}$.

For $i\geq i_0$, we infer from the lower-semi-continuity of the map $s_\varphi$, that the set
$$V_i=s_\varphi^{-1}(P_i)=\{y\in \beta Y:s_\varphi(y)\cap P_i\neq\emptyset\},$$
is open in $\beta Y$ and $V_i\supseteq V_{i+1}$.
For each $i\geq i_0$, we set
$$W_i=V_i\cup (\beta Y\setminus \overline{H_{n_i}}).$$
Clearly, $W_i$ is open in $\beta Y$ and $W_i\supseteq W_j$ for $i\leq j$ (because $V_i\supseteq V_j$ and $H_{n_i}\subseteq H_{n_j}$).
Moreover, by \eqref{10'} and \eqref{14''}, we have
$\widetilde{h}^{-1}(K_i)\subseteq W_i$, for every $i\geq i_0$. Fix an arbitrary $i\geq i_0$. If $j\geq i$, then $K_i\subseteq K_j$, so
$\widetilde{h}^{-1}(K_i)\subseteq \widetilde{h}^{-1}(K_j)\subseteq W_j$. If $i_0 \leq j<i$, then $\widetilde{h}^{-1}(K_i)\subseteq W_i\subseteq W_j$. Therefore for every $i\geq i_0$ we have
\begin{equation}
   \widetilde{h}^{-1}(K_i)\subseteq \bigcap_{j=i_0}^\infty W_j.  \label{11'}
\end{equation}

We claim that the
set $G'=\bigcap_{i=i_0}^\infty \widetilde{h}^\#(W_i)$ is the $G_\delta$-set we are looking for.
First, note that $G'$ is indeed a $G_\delta$-set in $bM$, by Proposition \ref{proposition hash} (a).
From \eqref{11'} and Proposition \ref{proposition hash} (c) we get
$$\bigcup_{i=i_0}^\infty K_i\subseteq \bigcap_{i=i_0}^\infty \widetilde{h}^\#(W_i).$$
It remains to show that $\bigcap_{i=i_0}^\infty \widetilde{h}^\#(W_i)\subseteq bM\setminus M$.
Suppose that this is not the case and fix $a\in M\cap\bigcap_{i=i_0}^\infty \widetilde{h}^\#(W_i)$.
Since the map $h:Y\to M$ is surjective, there is $y\in Y$ such that $h(y)=\widetilde{h}(y)=a$. By Proposition \ref{proposition hash} (b),
$\widetilde{h}^{-1}(a)\subseteq \bigcap_{i=i_0}^\infty W_i$.
Thus
$$y\in Y\cap \bigcap_{i=i_0}^\infty W_i=Y\cap \bigcap_{i=i_0}^\infty\left(V_i\cup (\beta Y\setminus \overline{H_{n_i}})\right).$$
On the other hand, it follows from Observation \ref{H1H2} that $y\in H_{n_i}$ for all but finitely many $i$'s. Hence we must have
$$y\in V_i=s_\varphi^{-1}(P_i),\mbox{ for all but finitely many }i\mbox{'s}.$$
In addition,
$y\in Y$ so the set $s_\varphi(y)=\supp_\varphi(y)$ is a finite subset of $X$ (cf. Lemmata \ref{s is lsc} and \ref{Lemma_supp}).
Therefore, there must be $x\in \supp_\varphi(y)\subseteq X$ such that the set $\{i:x\in P_i\}$ is infinite.
Since $P_{i_0}\supseteq P_{i_0+1}\supseteq\ldots$, we get $x\in X\cap P$, which contradicts \eqref{14''}.
\end{proof}

Let us remark that Theorem \ref{theorem Hurewicz} follows immediately from Theorem \ref{thrm proj Hurewicz}, Theorem \ref{theorem Velichko} and
\cite[Theorem 3.2]{Ko} (cf. \cite[Proposition 31]{BCM}).

Now, we present a proof of Theorem \ref{thrm proj Menger}. Conceptually the proof is virtually the same as the previous one.
It is more technical though. This is because in place of Theorem \ref{Hurewicz characterization} we need to use Theorem \ref{thrm_Telgarsky},
i.e., instead of dealing with $\sigma$-compact subsets of the remainder $bM\setminus M$  we need to work with strategies in the game
$kP(bM,bM\setminus M)$ which is a more complicated task.

\begin{proof}[Proof of Theorem \ref{thrm proj Menger}]
By symmetry, it is enough to show that the projective Menger property of $X$ implies that $Y$ is projectively Menger. To this end,
suppose that $X$ is projectively Menger and
let us
fix a continuous surjection
$h:Y\to M$ that maps $Y$ onto a separable metrizable space $M$. Let $bM$ be a metrizable compactification of $M$ and let
$\widetilde{h}:\beta Y\to bM$ be a continuous extension of $h$. Denote by $d$ a metric on $bM$ that generates the topology of $bM$. Note that

\begin{equation}
    \mbox{If }A\subseteq bM\setminus M, \mbox{ then }  \widetilde{h}^{-1}(A)\subseteq \beta Y\setminus Y. \label{9}
\end{equation}

In order to prove that $M$ is Menger, we will employ 
Theorem \ref{thrm_Telgarsky}. For this purpose, suppose that $\sigma$ is a strategy for player I in the $k$-Porada game $kP(bM,bM\setminus M)$.
We need to show that the strategy $\sigma$ is not winning.  Since $h$ is surjective, we have
$M\subseteq \bigcup_{n=1}^\infty\widetilde{h}(\overline{H_n})$, by Observation \ref{H1H2}.
So applying Proposition \ref{proposition modify strategy}, we may without loss of generality assume, that
\begin{equation}
\mbox{every compact set played according to }\sigma \mbox{ meets }  \bigcup_{n=1}^\infty\widetilde{h}(\overline{H_n}).  \label{11}
\end{equation}

Using $\sigma$, we will recursively define a strategy $\tau$ for player I in the $z$-Porada game $zP(\beta X,\beta X\setminus X)$
(cf. Proposition \ref{proposition proj Menger}).
In addition, with each open set $V_i$ played by player II in his $(i+1)$st move in the game
$zP(\beta X,\beta X\setminus X)$ (where the strategy $\tau$ is applied by player I),
we will associate a set $V'_i$ played by player II in his $(i+1)$st move in $kP(bM,bM\setminus M)$ (where the strategy $\sigma$ is applied by player I).

Let $(K_0,U_0)=\sigma(\emptyset)$ be the first move played by I according to $\sigma$. By \eqref{11}, there exists $n_0$ such that the compact set

$$K'_{n_0}=\widetilde{h}^{-1}(K_0)\cap \overline{H_{n_0}}$$
is nonempty.
Since $bM$ is metrizable, the set $K_0$ being compact, is $G_\delta$ in $bM$.
It follows from Proposition \ref{s is lsc} and Lemma \ref{G_delta via lsc map} that the set
$$G_0=s^{-1}_{\varphi^{-1}}\left(\widetilde{h}^{-1}(K_0)\right)=\{x \in \beta X: s_{\varphi^{-1}}(x)\cap \widetilde{h}^{-1}(K_0)\neq \emptyset\}$$
is a $G_\delta$-set in $\beta X$. In addition, by \eqref{9}, Lemma \ref{lemma_s=supp} and
Lemma \ref{Lemma_supp}, we have $G_0\subseteq \beta X\setminus X$.

The map $e_{n_0}$ restricted to $K'_{n_0}$ is lower semi-continuous (by Lemma \ref{lemma e_n is lsc}) and note that if $y\in K'_{n_0}$ then
$e_{n_0}(y)\subseteq G_0\subseteq \beta X\setminus X$ (by definition of $e_{n_0}$). Thus, we may consider the map
$e_{n_0}\upharpoonright K'_{n_0}$ as a (lower semi-continuous) map into $\mathcal{K}(G_0)$. By Theorem \ref{theorem_Bouziad},
this map admits a compact section, i.e., there is a compact set $L_0\subseteq G_0$ such that
$$L_0\cap e_{n_0}(y)\neq \emptyset,\mbox{ for every } y\in K'_{n_0}.$$
In particular, since $e_{n_0}(y)\subseteq s_\varphi(y)$, we have
\begin{equation}
L_0\cap s_{\varphi}(y)\neq \emptyset,\mbox{ for every } y\in K'_{n_0}.\label{10}
\end{equation} 

Using Lemma \ref{upgrading compact to zero}, we can enlarge the set $L_0$ to a zero-set in $\beta X$ contained in $G_0$.
Clearly, this is still a section of $e_{n_0}$, so without loss
of generality we can assume that each $L_0$ is a zero-set in $\beta X$.

We define
$$\tau(\emptyset)=(L_0,\beta X).$$

Let $V_0$ be the first move of player II in $zP(\beta X,\beta X\setminus X)$,
i.e., $V_0$ is an arbitrary open set in $\beta X$ containing
$L_0$.
Consider the following subset $W_0$ of $bM$:
$$W_0=\widetilde{h}^\#\left(s_{\varphi}^{-1}(V_0)\cup (\beta Y\setminus \overline{H_{n_0}})\right).$$
Since $s_\varphi$ is lower semi-continuous, it follows from Proposition \ref{proposition hash} (a), that $W_0$ is open in $bM$.
Moreover, since $L_0\subseteq V_0$, we infer from \eqref{10} that
$K'_{n_0}=\widetilde{h}^{-1}(K_0)\cap \overline{H_{n_0}}\subseteq s_{\varphi}^{-1}(V_0)$ and thus
$\widetilde{h}^{-1}(K_0)\subseteq s_{\varphi}^{-1}(V_0)\cup (\beta Y\setminus \overline{H_{n_0}})$.
Hence, by Proposition \ref{proposition hash} (c) we get $K_0\subseteq W_0$. Define
$$V'_0=W_0\cap U_0.$$
Clearly, $K_0\subseteq V'_0\subseteq U_0$, so $V'_0$ is a legal move of player II in $kP(bM,bM\setminus M)$. Let
$(K_1,U_1)=\sigma(V'_0)$ be the response of player I, consistent with her strategy. By \eqref{11} and Observation \ref{H1H2}, there is $n_1>n_0$ such that the compact set
$$K'_{n_1}=\widetilde{h}^{-1}(K_1)\cap \overline{H_{n_1}}.$$ 
is nonempty. Arguing as before, we note that the set
$$G_1=s^{-1}_{\varphi^{-1}}\left(\widetilde{h}^{-1}(K_1)\right)$$
is $G_\delta$ in $\beta X$ and $G_1\subseteq \beta X\setminus X$. Again, since $e_{n_1}(y)\subseteq G_1$, for $y\in K'_{n_1}$,
we infer that the map $e_{n_1}\upharpoonright K'_{n_1}$ (i.e., $e_{n_1}$ restricted to $K_{n_1}$) maps the compact set $K'_{n_1}$
lower semi-continuously into $\mathcal{K}(G_1)$. By Theorem \ref{theorem_Bouziad} this map admits a compact section $L_1$. Again, using
Lemma \ref{upgrading compact to zero}, we can assume that $L_1$ is a zero-set in $\beta X$.

We define $$\tau(V_0)=(L_1,V_0).$$

Let $V_1$ be an arbitrary open set in $\beta X$ satisfying $L_1\subseteq V_1\subseteq V_0$ (the next move of player II in
$zP(\beta X, \beta X\setminus X)$). We set
$$W_1=\widetilde{h}^\#\left(s_{\varphi}^{-1}(V_1)\cup (\beta Y\setminus \overline{H_{n_1}})\right).$$
The lower semi-continuity of $s_{\varphi}$ and Proposition \ref{proposition hash} (a) imply that $W_0$ is open in $bM$.
Arguing as before, we get that $K_1\subseteq W_1$. Let $V'_1$ be an open set in $bM$ satisfying
$$K_1\subseteq V'_1 \subseteq \overline{V'_1}\subseteq W_1\cap U_1.$$

We continue our construction following this pattern. In this way we define a strategy $\tau$ for player I in the game
$zP(\beta X,\beta X\setminus X)$. Moreover, a play
$$\tau(\emptyset),\; V_0,\;\tau(V_0),\;V_1,\; \tau(V_0,V_1),\ldots$$
in $zP(\beta X,\beta X\setminus X)$ generates the play
$$\sigma(\emptyset),\; V'_0,\;\sigma(V'_0),\;V'_1,\; \sigma(V'_0,V'_1),\ldots$$
in $kP(bM,bM\setminus M)$, where
\begin{align}
    &V'_k \subseteq \widetilde{h}^\#\left(s_{\varphi}^{-1}(V_k)\cup (\beta Y\setminus \overline{H_{n_k}}\right) \mbox{ and} \label{12}\\
    &\overline{V'_{k+1}}\subseteq V'_{k} \label{13}
\end{align}
The numbers $n_0<n_1<\ldots <n_k<\ldots$ form an increasing sequence.

By our assumption, the space $X$ is projectively Menger, hence by Proposition \ref{proposition proj Menger}, there is a play
$$\tau(\emptyset),\; V_0,\;\tau(V_0),\;V_1,\; \tau(V_0,V_1),\ldots$$
in which player I applies her strategy $\tau$ and fails, i.e
\begin{equation}
 \emptyset\neq\bigcap_{k=0}^\infty V_k\subseteq \beta X\setminus X. \label{14}
\end{equation}
The above play generates the play
$$\sigma(\emptyset),\; V'_0,\;\sigma(V'_0),\;V'_1,\; \sigma(V'_0,V'_1),\ldots$$
in $kP(bM,bM\setminus M)$. We claim that player II wins this run of the game and thus $\sigma$ is not winning for player I.

Indeed, otherwise $M\cap\bigcap_{k=0}^\infty{V'_k}\neq \emptyset$ (note that \eqref{13} guarantees that the intersection of the family
$\{V'_k:k=0,1,\ldots\}$ is nonempty by compactness).
Fix $a\in M\cap\bigcap_{k=0}^\infty{V'_k}$. Since $h:Y\to M$ is surjective there is $y\in Y$ with $h(y)=a$.
Applying \eqref{12} and Proposition \ref{proposition hash} (b), we get
$$\widetilde{h}^{-1}(a)\subseteq s_{\varphi}^{-1}(V_k)\cup \left(\beta Y\setminus \overline{H_{n_k}}\right),$$
for every $k$.
On the other hand, it follows from Observation \ref{H1H2} that $y\in H_{n_k}$ for all but finitely many $k$'s. Hence, we must have
$$y\in s_{\varphi}^{-1}(V_k),\mbox{ for all but finitely many }k\mbox{'s}.$$
In addition,
$y\in Y$ so the set $s_\varphi(y)=\supp_\varphi(y)$ is a finite subset of $X$ (cf. Lemmata \ref{s is lsc} and \ref{Lemma_supp}).
Therefore, there must be $x\in \supp_\varphi(y)\subseteq X$ such that the set $\{k:x\in V_k\}$ is infinite.
Since $V_0\supseteq V_1\supseteq\ldots$, we get $x\in X\cap\bigcap_{k=0}^\infty V_k$, which contradicts \eqref{14}
\end{proof} 

\begin{proof}[Proof of Theorem \ref{theorem Menger}]
By symmetry, it is enough to show that the Menger property of $X$ implies the Menger property of $Y$. Suppose that $X$ is Menger.
Then $X$ is Lindel\"of, so by Velichko's Theorem \ref{theorem Velichko}, the space $Y$ is Lindel\"of too. Moreover, since $X$ is Menger it is projectively
Menger (cf. Proposition \ref{proposition proj_Menger}), so according to Theorem \ref{thrm proj Menger}, the space $Y$ is projectively Menger. The result follows now
from Proposition \ref{proposition proj_Menger}. 
\end{proof}

\section*{Acknowledgements}
The author was partially supported by the NCN (National Science Centre, Poland) research Grant no. 2020/37/B/ST1/02613
\bibliographystyle{siam}
\bibliography{bib.bib}
\end{document}